\documentclass[10pt]{amsart}
\usepackage{amscd,amsmath,amssymb,amsthm,enumerate,graphicx}
\usepackage[bookmarks=false,dvipdfm]{hyperref}
\newtheorem{definition}[equation]{Definition}
\newtheorem{lemma}[equation]{Lemma}
\newtheorem{proposition}[equation]{Proposition}
\newtheorem{theorem}[equation]{Theorem}
\newtheorem{corollary}[equation]{Corollary}
\newtheorem{remark}[equation]{Remark}

\newtheorem{problem}[equation]{Problem}

\newcommand\lemmaref[1]{Lemma~\ref{#1}}
\newcommand\propositionref[1]{Proposition~\ref{#1}}
\newcommand\theoremref[1]{Theorem~\ref{#1}}
\newcommand\corollaryref[1]{Corollary~\ref{#1}}
\newcommand\remarkref[1]{Remark~\ref{#1}}

\newcommand\problemref[1]{Problem~\ref{#1}}
\title{A criterion for discrete branching laws for Klein four symmetric pairs and its application to $E_{6(-14)}$}
\author{Haian HE}
\date{}
\address{Department of Mathematics, College of Sciences, Shanghai University, No. 99 Shangda Road, Baoshan District, Shanghai, 200444 China P. R.}
\email{hebe.hsinchu@yahoo.com.tw}
\subjclass[2010]{22E46}
\keywords{Klein four symmetric pair; branching law; discretely decomposable; unitarizable simple $(\mathfrak{g},K)$-module; associated variety}
\begin{document}
\begin{abstract}
Let $G$ be a noncompact connected simple Lie group, and $(G,G^\Gamma)$ a Klein four symmetric pair. In this paper, the author shows a necessary condition for the discrete decomposability of unitarizable simple $(\mathfrak{g},K)$-modules for Klein for symmetric pairs. Precisely, if certain conditions hold for $(G,G^\Gamma)$, there does not exist any unitarizable simple $(\mathfrak{g},K)$-module that is discretely decomposable as a $(\mathfrak{g}^\Gamma,K^\Gamma)$-module. As an application, for $G=\mathrm{E}_{6(-14)}$, the author obtains a complete classification of Klein four symmetric pairs $(G,G^\Gamma)$ with $G^\Gamma$ noncompact, such that there exists at least one nontrivial unitarizable simple $(\mathfrak{g},K)$-module that is discretely decomposable as a $(\mathfrak{g}^\Gamma,K^\Gamma)$-module and is also discretely decomposable as a $(\mathfrak{g}^\sigma,K^\sigma)$-module for some nonidentity element $\sigma\in\Gamma$.
\end{abstract}
\maketitle
\section{introduction}
Let $G$ be a noncompact connected simple Lie group, and $\Gamma$ a Klein four subgroup of the automorphism group $\mathrm{Aut}G$ of $G$. Then $(G,G^\Gamma)$ forms a Klein four symmetric pair, where $G^\Gamma$ is the subgroup of the fixed points under action of all elements in $\Gamma$ on $G$. Write $\mathfrak{g}$ for the Lie algebra of $G$, and $\mathfrak{g}^\Gamma$ for the Lie subalgebra of $\mathfrak{g}$ corresponding to $G^\Gamma$. Take a $\Gamma$-stable maximal compact subgroup $K$ of $G$ so that $K^\Gamma=K\cap G^\Gamma$ is a maximal compact subgroup of $G^\Gamma$. For a $(\mathfrak{g},K)$-module $X$, it can be regarded as a $(\mathfrak{g}^\Gamma,K^\Gamma)$-module by restriction. A fundamental problem is to classify all the Klein four symmetric pairs $(G,G^\Gamma)$ such that there exists at least one nontrivial unitarizable simple $(\mathfrak{g},K)$-module that is discretely decomposable as a $(\mathfrak{g}^\Gamma,K^\Gamma)$-module. If $G^\Gamma$ is compact, the problem is trivial because any unitarizable simple $(\mathfrak{g},K)$-module is discretely decomposable as a $(\mathfrak{g}^\Gamma,K^\Gamma)$-module. Hence, the author is only interested in the case when $G^\Gamma$ is noncompact.

When $G$ is supposed to be an exceptional simple Lie group of Hermitian type, the discretely decomposable restrictions of unitarizable simple $(\mathfrak{g},K)$-modules for Klein four symmetric pairs ware studied in \cite{H1}, \cite{H2}, and \cite{H3}. In this case, one may define Klein four symmetric pairs of holomorphic type. The discretely decomposable condition for Klein four symmetric pairs of holomorphic type is easy to deal with because any highest / lowest weight simple $(\mathfrak{g},K)$-module is discretely decomposable as a $(\mathfrak{g}^\Gamma,K^\Gamma)$-module. As for Klein four symmetric pairs of non-holomorphic type, it is difficult. In \cite{H3}, the author classified the Klein four symmetric pairs $(G,G^\Gamma)$ for exceptional simple Lie groups of Hermitian type such that there exists at least one nontrivial unitarizable simple $(\mathfrak{g},K)$-module $X$ that is both discretely decomposable as a $(\mathfrak{g}^\Gamma,K^\Gamma)$-module and is discretely decomposable as a $(\mathfrak{g}^\sigma,K^\sigma)$-module for some nonidentity element $\sigma\in\Gamma$ of anti-holomorphic type. Here, $\mathfrak{g}^\sigma:=\{X\in\mathfrak{g}\mid\sigma X=X\}$.

Suppose that $\Gamma$ is generated by the two involutive automorphisms $\sigma$ and $\tau$, and then $\Gamma=\{1,\sigma,\tau,\sigma\tau\}$. The nonidentity elements in $\Gamma$ define three symmetric pairs $(G,G^\sigma)$, $(G,G^\tau)$, and $(G,G^{\sigma\tau})$ for $G$, where $G^\sigma$ (respectively, $G^\tau$, or $G^{\sigma\tau}$ ) is the subgroup of the fixed points under the action of $\sigma$ (respectively, $\tau$, or $\sigma\tau$) on $G$. The branching rules for symmetric pairs were studied by many mathematicians, among which the classification of symmetric pairs $(G,G')$ such that there exists at least one nontrivial unitarizable simple $(\mathfrak{g},K)$-module which is discretely decomposable as a $(\mathfrak{g}',K')$-module was completely solved in \cite[Theorem 5.2 \& Corollary 5.8]{KO2} based on the criteria established in \cite{Ko2}, \cite{Ko3}, and \cite{Ko4}, where $\mathfrak{g}'$ is the Lie subalgebra of $\mathfrak{g}$ corresponding to $G'$.

Associated varieties are useful tools to study the discrete decomposability of the restrictions of unitarizable simple $(\mathfrak{g},K)$-modules. Roughly speaking, for a simple Lie group $G$ and its reductive subgroup $G'$, if $X$ is a unitarizable simple $(\mathfrak{g},K)$-module which is discretely decomposable as a $(\mathfrak{g}',K')$-module, and if $Y$ is a nonzero simple $(\mathfrak{g}',K')$-submodule of $X$, then the projection of the associated variety of $X$ from $\mathfrak{g}$ to $\mathfrak{g}'$ is contained in the associated variety of $Y$. Thus, the associated varieties offer a necessary condition for the discretely decomposable restriction of a unitarizable simple $(\mathfrak{g},K)$-module. In \cite{KO2}, the authors just made use of associated varieties to find a crucial necessary condition for discretely decomposable restrictions for symmetric pairs. As the natural generalization of symmetric pairs, the author studies Klein four symmetric pairs.

The classification of Klein four symmetric pairs is far from clear as that of symmetric pairs, so it is much hard to study their branching laws. However, the classification of Klein four subgroups of automorphism groups of compact Lie algebras was done in \cite{HY}, and a more general classification of elementary abelian $2$-subgroups of automorphism groups of compact Lie algebras was done in \cite{Y}. The author will make use of the classification results to study Klein four symmetric pairs.

The article is organized as follows. After a quick review of associated varieties, noncompact maximal roots, and minimal orbits, the author will recall an important result of discretely decomposable restrictions of unitarizable simple $(\mathfrak{g},K)$-modules for symmetric pairs in Section 2. Then in Section 3, the author will show the main result in this article, which offers a necessary condition for the existence of unitarizable simple $(\mathfrak{g},K)$-modules whose restrictions are discretely decomposable restrictions for Klein four symmetric pairs, see \theoremref{10} and \corollaryref{12} below. Finally, the author will apply the main result to $\mathrm{E}_{6(-14)}$, and will solve a branching problem for Klein four symmetric pairs of $\mathrm{E}_{6(-14)}$, which is another important result in this article, see \theoremref{18} below, and is the complete answer to the problem raised in \cite{H3} for $\mathrm{E}_{6(-14)}$.

The author is supported by National Natural Science Foundation of China (Grant Number: 11901378).
\section{Preliminary}
\subsection{Associated varieties and minimal orbits}
Let $G$ be a noncompact connected simple Lie group with Lie algebra $\mathfrak{g}$. Fix a Cartan decomposition $\mathfrak{g}=\mathfrak{k}+\mathfrak{p}$, write $\mathfrak{g}_\mathbb{C}=\mathfrak{k}_\mathbb{C}+\mathfrak{p}_\mathbb{C}$ for its complexification, $\mathfrak{g}_\mathbb{C}^*=\mathfrak{k}_\mathbb{C}^*+\mathfrak{p}_\mathbb{C}^*$ for the dual space, and $K$ for the connected subgroup of $G$ with Lie algebra $\mathfrak{k}$. Denote by $K_\mathbb{C}$ the subgroup of the inner automorphism group $\mathrm{Int}\mathfrak{g}_\mathbb{C}$ of $\mathfrak{g}_\mathbb{C}$ generated by $\mathrm{exp}(\mathrm{ad}\mathfrak{k}_\mathbb{C})$. The adjoint group $K_\mathbb{C}$ acts canonically on $\mathfrak{p}_\mathbb{C}$ and on the dual space $\mathfrak{p}_\mathbb{C}^*$. Take a Cartan subalgebra $\mathfrak{t}$ of $\mathfrak{k}$ and choose a positive system $\Delta^+(\mathfrak{k}_\mathbb{C},\mathfrak{t}_\mathbb{C})$. Similarly, denote by $\mathfrak{t}^*$ the dual space of $\mathfrak{t}$.

Recall the definition for the maximal noncompact root as in \cite[Definition 2.1]{KO2}. If $G$ is not of Hermitian type, then $K_\mathbb{C}$ acts irreducibility on $\mathfrak{p}_\mathbb{C}$, in which case the maximal noncompact root $\beta\in\sqrt{-1}\mathfrak{t}^*$ is defined to be the highest weight in $\mathfrak{p}_\mathbb{C}^*$. If $G$ is of Hermitian type, then $\mathfrak{p}_\mathbb{C}=\mathfrak{p}_++\mathfrak{p}_-$ is an irreducible decomposition as a $K_\mathbb{C}$-representation, in which case the maximal noncompact root $\beta\in\sqrt{-1}\mathfrak{t}^*$ is defined to be the highest weight in $\mathfrak{p}_+^*$ where $\mathfrak{p}_\pm^*$ is the dual space of $\mathfrak{p}_\pm$. Notice that in either case, $-\beta$ is also a weight in $\mathfrak{p}_\mathbb{C}^*$; in particular, $-\beta$ is a weight in $\mathfrak{p}_-^*$ for $\mathfrak{g}$ of Hermitian type. The weight spaces $\mathfrak{p}_\beta^*$ and $\mathfrak{p}_{-\beta}^*$ are $1$-dimensional.

Let $U(\mathfrak{g}_\mathbb{C})$ be the universal enveloping algebra of $\mathfrak{g}_\mathbb{C}$ with the standard increasing filtration of $\{U_j(\mathfrak{g}_\mathbb{C})\}_{j\in\mathbb{Z}_{\geq0}}$. Suppose that $X$ is a finitely generated $\mathfrak{g}_\mathbb{C}$-module. A filtration $\{X_i\}_{i\in\mathbb{Z}_{\geq0}}$ of $X$ is called a good filtration if it satisfies the following conditions:
\begin{itemize}
\item $X=\bigcup\limits_{i\in\mathbb{Z}_{\geq0}}X_i$;
\item $X_i$ is a finite dimensional subspace for any $i\in\mathbb{Z}_{\geq0}$;
\item $U_j(\mathfrak{g}_\mathbb{C})X_i\subseteq X_{i+j}$ for any $i,j\in\mathbb{Z}_{\geq0}$;
\item there exists $n\in\mathbb{Z}_{\geq0}$ such that $U_j(\mathfrak{g}_\mathbb{C})X_i=X_{i+j}$ for any $i\geq n$ and $j\in\mathbb{Z}_{\geq0}$.
\end{itemize}
The graded algebra $\mathrm{gr}U(\mathfrak{g}_\mathbb{C}):=\bigoplus\limits_{j\in\mathbb{Z}_{\geq0}}U_j(\mathfrak{g}_\mathbb{C})/U_{j-1}(\mathfrak{g}_\mathbb{C})$ is isomorphic to the symmetric algebra $S(\mathfrak{g}_\mathbb{C})$, and $\mathrm{gr}X:=\bigoplus\limits_{i\in\mathbb{Z}_{\geq0}}X_i/X_{i-1}$ forms a graded $S(\mathfrak{g}_\mathbb{C})$-module. Let $\mathrm{Ann}_{S(\mathfrak{g}_\mathbb{C})}(\mathrm{gr}X):=\{f\in S(\mathfrak{g}_\mathbb{C})\mid fv=0\textrm{ for any }v\in\mathrm{gr}X\}$, and the associated variety of $X$ is defined to be $\mathcal{V}_{\mathfrak{g}_\mathbb{C}}(X):=\{x\in\mathfrak{g}_\mathbb{C}^*\mid f(x)=0\textrm{ for any }f\in\mathrm{Ann}_{S(\mathfrak{g}_\mathbb{C})}(\mathrm{gr}X)\}$, which does not depend on the choice of good filtration.

Let $\mathcal{N}(\mathfrak{g}_\mathbb{C}^*)$ be the nilpotent cone of $\mathfrak{g}_\mathbb{C}^*$, and set $\mathcal{N}(\mathfrak{p}_\mathbb{C}^*):=\mathcal{N}(\mathfrak{g}_\mathbb{C}^*)\cap\mathfrak{p}_\mathbb{C}^*$. It is well known that $\mathcal{V}_{\mathfrak{g}_\mathbb{C}}(X)\subseteq\mathcal{N}(\mathfrak{g}_\mathbb{C}^*)$. Identity $\mathfrak{g}_\mathbb{C}^*$ with $\mathfrak{g}_\mathbb{C}$ by the Killing form, and the adjoint action of $K_\mathbb{C}$ on $\mathfrak{g}_\mathbb{C}$ induces the coadjoint action on $\mathfrak{g}_\mathbb{C}^*$.
\begin{proposition}\label{7}
Let $G$ be a noncompact connected simple Lie group.
\begin{enumerate}[(1)]
\item If $G$ is not of Hermitian type, then there is a unique minimal $K_\mathbb{C}$-orbit in $\mathcal{N}(\mathfrak{p}_\mathbb{C}^*)$, which is given by $K_\mathbb{C}\cdot(\mathfrak{p}_\beta^*\setminus\{0\})$.
\item If $G$ is of Hermitian type, then there are two minimal $K_\mathbb{C}$-orbits in $\mathcal{N}(\mathfrak{p}_\mathbb{C}^*)$, which are given by $K_\mathbb{C}\cdot(\mathfrak{p}_\beta^*\setminus\{0\})$ and $K_\mathbb{C}\cdot(\mathfrak{p}_{-\beta}^*\setminus\{0\})$. They have the same dimension.
\end{enumerate}
\end{proposition}
\begin{proof}
See \cite[Proposition 2.2]{KO2}.
\end{proof}
\begin{remark}\label{11}
It is known that the associated variety $\mathcal{V}_{\mathfrak{g}_\mathbb{C}}(X)$ is a $K_\mathbb{C}$-stable closed subset of $\mathfrak{p}_\mathbb{C}^*$. Then it follows from \propositionref{7} that $K_\mathbb{C}\cdot\mathfrak{p}_\beta^*\subseteq\mathcal{V}_{\mathfrak{g}_\mathbb{C}}(X)$ if $G$ is not of Hermitian type, and that $K_\mathbb{C}\cdot\mathfrak{p}_\beta^*\subseteq\mathcal{V}_{\mathfrak{g}_\mathbb{C}}(X)$ or $K_\mathbb{C}\cdot\mathfrak{p}_{-\beta}^*\subseteq\mathcal{V}_{\mathfrak{g}_\mathbb{C}}(X)$ if $G$ is of Hermitian type.
\end{remark}
\subsection{Discrete decomposability for symmetric pairs}
Let $G'$ be a reductive subgroup of $G$ with the Lie algebra $\mathfrak{g}'$. Take a maximal compact subgroup $K$ of $G$, which is defined by a Cartan involution $\theta$ on $G$. In particular, it is assumed that $\theta(G')=G'$; equivalently, $K':=K\cap G'$ is a maximal compact subgroup of $G'$.
\begin{definition}\label{1}
A $(\mathfrak{g},K)$-module $X$ is called discretely decomposable as a $(\mathfrak{g}',K')$-module if there exists an increasing filtration $\{X_i\}_{i\in\mathbb{Z}^+}$ of $(\mathfrak{g}',K')$-modules such that $\bigcup\limits_{i\in\mathbb{Z}^+}X_i=X$ and each $X_i$ is of finite length as a $(\mathfrak{g}',K')$-module for any $i\in\mathbb{Z}^+$.
\end{definition}
\begin{proposition}\label{3}
Let $X$ be a unitarizable simple $(\mathfrak{g},K)$-module. Then the following conditions are equivalent.
\begin{enumerate}[(1)]
\item $X$ is discretely decomposable as a $(\mathfrak{g}^\Gamma,K')$-module;
\item there exists a simple $(\mathfrak{g}',K')$-module $Y$ such that $\mathrm{Hom}_{(\mathfrak{g}',K')}(Y,X)\neq\{0\}$;
\item $X$ is isomorphic to an algebraic direct sum of simple $(\mathfrak{g}',K')$-modules.
\end{enumerate}
\end{proposition}
\begin{proof}
See \cite[Lemma 1.3 \& Lemma 1.5]{Ko4}.
\end{proof}
The natural embedding of the Lie algebras $\mathfrak{g}'\hookrightarrow\mathfrak{g}$ gives a projection of the complexified dual spaces $\mathrm{pr}_{\mathfrak{g}\rightarrow\mathfrak{g}'}:\mathfrak{g}_\mathbb{C}^*\twoheadrightarrow\mathfrak{g'}_\mathbb{C}^*$.
\begin{proposition}\label{8}
Let $X$ be a $(\mathfrak{g},K)$-module of finite length. If $X$ is discretely decomposable as a $(\mathfrak{g}',K')$-module, then $\mathrm{pr}_{\mathfrak{g}\rightarrow\mathfrak{g}'}\mathcal{V}_\mathbb{C}(X)\subseteq\mathcal{N}(\mathfrak{g'}_\mathbb{C}^*)$, where $\mathcal{N}(\mathfrak{g'}_\mathbb{C}^*)$ is the nilpotent cone of $\mathfrak{g'}_\mathbb{C}^*$.
\end{proposition}
\begin{proof}
See \cite[Corollary 3.4]{Ko4}.
\end{proof}
Now let $(G,G^\sigma)$ be a symmetric pair, where $\sigma\in\mathrm{Aut}G$ is an involutive automorphism which commutes with $\theta$. Then $\sigma$ stabilizes $K$, and $K^\sigma=K\cap G^\sigma$ is a maximal compact subgroup of $G^\sigma$. Without confusion, use the same symbol $\sigma$ for its differential on the Lie algebra $\mathfrak{g}$ of $G$, which stabilizes the Lie subalgebra $\mathfrak{k}$ corresponding to $K$. Take a $\sigma$-stable Cartan subalgebra $\mathfrak{t}$ of $\mathfrak{k}$ and a positive root system $\Delta^+(\mathfrak{k}_\mathbb{C},\mathfrak{t}_\mathbb{C})$. Identity $\mathfrak{g}_\mathbb{C}^*$ with $\mathfrak{g}_\mathbb{C}$ by the Killing form, and $\sigma$ can be regarded as an involutive automorphism of $\mathfrak{g}_\mathbb{C}^*$.

Remember the definition for the maximal noncompact root $\beta$ mentioned in Subsection 2.1.
\begin{proposition}\label{9}
Let $G$ be a noncompact connect simple Lie group, and $(G,G^\sigma)$ a symmetric pair. Then there exists a nontrivial unitarizable simple $(\mathfrak{g},K)$-module $X$ such that $X$ is discretely decomposable as a $(\mathfrak{g}^\sigma,K^\sigma)$-module if and only if $\sigma\beta\neq-\beta$.
\end{proposition}
\begin{proof}
See \cite[Theorem 5.2 \& Corollary 5.8]{KO2}.
\end{proof}
\section{Main Results for Klein four symmetric pairs}
Recall the definition for Klein four symmetric pairs.
\begin{definition}\label{4}
Let $G$ (respectively, $\mathfrak{g}$) be a real simple Lie group (respectively, Lie algebra), and let $\Gamma$ be a Klein four subgroup of $\mathrm{Aut}G$ (respectively, $\mathrm{Aut}\mathfrak{g}$). Denote by $G^\Gamma$ (respectively, $\mathfrak{g}^\Gamma$) the subgroup (respectively, subalgebra) of the fixed points under the action of all elements in $\Gamma$ on $G$ (respectively, $\mathfrak{g}$). Then $(G,G^\Gamma)$ (respectively, $(\mathfrak{g},\mathfrak{g}^\Gamma)$) is called a Klein four symmetric pair. In particular, if $G$ (respectively, $\mathfrak{g}$) is a real simple Lie group (respectively, Lie algebra) of Hermitian type, and every nonidentity element $\sigma\in\Gamma$ defines a symmetric pair of holomorphic type, then $(G,G^\Gamma)$ (respectively, $(\mathfrak{g},\mathfrak{g}^\Gamma)$) is called a Klein four symmetric pair of holomorphic type.
\end{definition}
For a Klein four symmetric pair $(G,G^\Gamma)$, the maximal compact subgroup $K$ of $G$ is always supposed to be $\Gamma$-stable in the sense that $\sigma(K)=K$ for all $\sigma\in\Gamma$, so that $K^\Gamma=K\cap G^\Gamma$ is a maximal compact subgroup of $G^\Gamma$. Also, fix a $\Gamma$-stable Cartan subalgebra $\mathfrak{t}$ of $\mathfrak{k}$. Similar to the case of symmetric pairs, $\Gamma$ acts on the dual space $\mathfrak{g}_\mathbb{C}^*$ which is identified with $\mathfrak{g}_\mathbb{C}$ by the Killing form.
\begin{proposition}\label{6}
If $G$ is a simple Lie group of Hermitian type and $(G,G^\Gamma)$ is a Klein four symmetric pair of holomorphic type, then any irreducible unitary highest / lowest weight $(\mathfrak{g},K)$-module $X$ is discretely decomposable as a $(\mathfrak{g}^\Gamma,K^\Gamma)$-module and is discretely decomposable as a $(\mathfrak{g}^\sigma,K^\sigma)$-module for any $\sigma\in\Gamma$.
\end{proposition}
\begin{proof}
The conclusion follows from \cite[Theorem 2.9(1) \& Example 2.13 \& Example 3.3]{Ko3} and \cite[Proposition 1.6]{Ko4} immediately.
\end{proof}
Now consider the discrete decomposability of the restrictions for Klein four symmetric pairs for general cases. Recall \propositionref{9} that $\sigma\beta\neq-\beta$ is an equivalent condition for the existence of unitarizable simple $(\mathfrak{g},K)$-module whose restriction is discretely decomposable for a symmetric pair $(G,G^\sigma)$. The following lemma plays a crucial role in the main result of this article.
\begin{lemma}\label{2}
Let $(G,G^\Gamma)$ be a Klein four symmetric pair. If there are two involutive automorphisms $\sigma$ and $\tau$ in $\Gamma$ such that $\sigma\beta=-\tau\beta\neq\pm\beta$, then $\mathrm{pr}_{\mathfrak{g}\rightarrow\mathfrak{g}^\Gamma}(K_\mathbb{C}\cdot\mathfrak{p}_\beta^*)\nsubseteq\mathcal{N}(\mathfrak{g}_\mathbb{C}^{\Gamma*})$.
\end{lemma}
\begin{proof}
Obviously, $\sigma\tau\beta=-\beta$ because of $\sigma\beta=-\tau\beta$. Take a nonzero element $x\in\mathfrak{p}_\beta^*$, and one has $\sigma\tau(x)\in\mathfrak{p}_{-\beta}^*$. Let $\overline{x}$ denote the complex conjugation of $x$ with respect to the real form $\mathfrak{g}$ of $\mathfrak{g}_\mathbb{C}$, and one also has $\overline{x}\in\mathfrak{p}_{-\beta}^*$. Hence, $\overline{\sigma\tau(x)}\in\mathfrak{p}_\beta^*$ and $y:=x+\overline{\sigma\tau(x)}\in\mathfrak{p}_\beta^*$. Replace $x$ by $cx$ for some $c\in\mathbb{C}$ if necessary, $y$ can be always assumed to be nonzero. It is known that neither $\sigma(y)$ nor $\tau(y)$ is in $\mathfrak{p}_{\pm\beta}^*$ because $\sigma\beta=-\tau\beta\neq\pm\beta$, and hence $\mathrm{pr}_{\mathfrak{g}\rightarrow\mathfrak{g}^\Gamma}(y)=\frac14(y+\sigma(y)+\tau(y)+\sigma\tau(y))$ is nonzero. Moreover, one has  $\mathrm{pr}_{\mathfrak{g}\rightarrow\mathfrak{g}^\Gamma}(y)=\frac14(x+\sigma(x)+\tau(x)+\sigma\tau(x)+\overline{x}+\overline{\sigma(x)}+\overline{\tau(x)}+ \overline{\sigma\tau(x)})\in\mathfrak{p}_\mathbb{C}^*\cap\mathfrak{g}^*=\mathfrak{p}^*$ which is a nonzero semisimple element. Therefore, $\mathrm{pr}_{\mathfrak{g}\rightarrow\mathfrak{g}^\Gamma}(y)\notin\mathcal{N}(\mathfrak{g}_\mathbb{C}^{\Gamma*})$ and hence $\mathrm{pr}_{\mathfrak{g}\rightarrow\mathfrak{g}^\Gamma}(K_\mathbb{C}\cdot\mathfrak{p}_\beta^*)\nsubseteq\mathcal{N}(\mathfrak{g}_\mathbb{C}^{\Gamma*})$.
\end{proof}
\begin{remark}\label{5}
The condition $\sigma\beta=-\tau\beta\neq\pm\beta$ in \lemmaref{2} is equivalent to $\sigma(-\beta)=-\tau(-\beta)\neq\pm\beta$, and hence one also obtains that $\mathrm{pr}_{\mathfrak{g}\rightarrow\mathfrak{g}^\Gamma}(K_\mathbb{C}\cdot\mathfrak{p}_{-\beta}^*)\nsubseteq\mathcal{N}(\mathfrak{g}_\mathbb{C}^{\Gamma*})$.
\end{remark}
The author may state the main theorem now.
\begin{theorem}\label{10}
Let $(G,G^\Gamma)$ be a Klein four symmetric pair. If there are two involutive automorphisms $\sigma$ and $\tau$ in $\Gamma$ such that $\sigma\beta=-\tau\beta\neq\pm\beta$, then there does not exist any nontrivial unitarizable simple $(\mathfrak{g},K)$-module that is discretely decomposable as a $(\mathfrak{g}^\Gamma,K^\Gamma)$-module.
\end{theorem}
\begin{proof}
From \remarkref{11}, one has either $K_\mathbb{C}\cdot\mathfrak{p}_\beta^*\subseteq\mathcal{V}_{\mathfrak{g}_\mathbb{C}}(X)$ or $K_\mathbb{C}\cdot\mathfrak{p}_{-\beta}^*\subseteq\mathcal{V}_{\mathfrak{g}_\mathbb{C}}(X)$ for any nontrivial unitarizable simple $(\mathfrak{g},K)$-module $X$. On the other hand, by \lemmaref{2} and \remarkref{5}, $\mathrm{pr}_{\mathfrak{g}\rightarrow\mathfrak{g}^\Gamma}(K_\mathbb{C}\cdot\mathfrak{p}_{\pm\beta}^*)\nsubseteq\mathcal{N}(\mathfrak{g}_\mathbb{C}^{\Gamma*})$. Therefore, $\mathrm{pr}_{\mathfrak{g}\rightarrow\mathfrak{g}^\Gamma}\mathcal{V}_{\mathfrak{g}_\mathbb{C}}(X)\nsubseteq\mathcal{N}(\mathfrak{g}_\mathbb{C}^{\Gamma*})$ for any nontrivial unitarizable simple $(\mathfrak{g},K)$-module $X$, and the conclusion follows from \propositionref{8}.
\end{proof}
\begin{corollary}\label{12}
Let $G$ be a simple Lie group, and $(G,G^\Gamma)$ a Klein four symmetric pair defined by the Klein four subgroup $\Gamma=\langle\sigma,\tau\rangle$ generated by two involutive automorphisms $\sigma$ and $\tau$. Suppose that there exists a nontrivial unitarizable simple $(\mathfrak{g},K)$-module that is discretely decomposable as a $(\mathfrak{g}^\Gamma,K^\Gamma)$-module. Suppose that
\begin{enumerate}[(1)]
\item there exists a nontrivial unitarizable simple $(\mathfrak{g},K)$-module that is discretely decomposable as a $(\mathfrak{g}^\sigma,K^\sigma)$-module;
\item there exists a nontrivial unitarizable simple $(\mathfrak{g},K)$-module that is discretely decomposable as a $(\mathfrak{g}^\tau,K^\tau)$-module;
\item there does not exist any nontrivial unitarizable simple $(\mathfrak{g},K)$-module that is discretely decomposable as a $(\mathfrak{g}^{\sigma\tau},K^{\sigma\tau})$-module.
\end{enumerate}
Then there does not exist any nontrivial unitarizable simple $(\mathfrak{g},K)$-module that is discretely decomposable as a $(\mathfrak{g}^\Gamma,K^\Gamma)$-module.
\end{corollary}
\begin{proof}
The third condition implies that $\sigma\tau\beta=-\beta$ by \propositionref{9}, and hence $\sigma\beta=-\tau\beta$. Again by \propositionref{9}, the first condition and the second condition imply that $\sigma\beta=-\tau\beta\neq\pm\beta$. The conclusion follows from \theoremref{10}.
\end{proof}
\section{Applications to $\mathrm{E}_{6(-14)}$}
\subsection{Problem and strategy}
Let $(G,G^\Gamma)$ be a Klein four symmetric pair. Retain all the settings as previous sections. Consider the following problem.
\begin{problem}\label{14}
Classify all the Klein four symmetric pairs $(G,G^\Gamma)$ such that there exists at least one nontrivial unitarizable simple $(\mathfrak{g},K)$-module $X$ satisfying two conditions:
\begin{enumerate}[$\bullet$]
\item $X$ is discretely decomposable as a $(\mathfrak{g}^\Gamma,K^\Gamma)$-module;
\item $X$ is discretely decomposable as a $(\mathfrak{g}^\sigma,K^\sigma)$-module for some nonidentity element $\sigma\in\Gamma$.
\end{enumerate}
\end{problem}
\begin{remark}\label{15}
\problemref{14} does make sense. In fact, if $X$ is discretely decomposable as a $(\mathfrak{g}^\Gamma,K^\Gamma)$-module and $\sigma\in\Gamma$, it is not known whether $X$ is discretely decomposable as a $(\mathfrak{g}^\sigma,K^\sigma)$-module. For example, let $G=\mathrm{SL}(n,\mathbb{C})$. Suppose that $\Gamma=\langle\sigma,\tau\rangle$ is the Klein four subgroup generated $\sigma$ and $\tau$, where $\sigma:A\mapsto\overline{A}$ is the natural complex conjugation and $\tau:A\mapsto (A^{-1})^T$ is the composition of inverse and transposition. Then $G^\Gamma=\mathrm{SO}(n)$ is compact and hence any unitarizable simple $(\mathfrak{g},K)$-module is discretely decomposable as a $(\mathfrak{g}^\Gamma,K^\Gamma)$-module, but $G^\sigma=\mathrm{SL}(n,\mathbb{R})$ is a split real form of $G$ and by \cite[Theorem 8.1]{Ko4}, there is no unitarizable simple $(\mathfrak{g},K)$-module discretely decomposable as a $(\mathfrak{g}^\sigma,K^\sigma)$-module.
\end{remark}
\begin{remark}\label{16}
It is obvious that if $G$ is a simple Lie group of Hermitian type and $(G,G^\Gamma)$ is a Klein four symmetric pair of holomorphic type, then any highest / lowest weight simple $(\mathfrak{g},K)$-module satisfies the requirements in \problemref{14}. Moreover, for exceptional Lie groups of Hermitian type, \problemref{14} was discussed under the assumption that $\sigma$ is an involutive automorphism of anti-holomorphic type in \cite{H3}. In this section, the author considers the case when $(G,G^\Gamma)$ is a Klein four symmetric pair of non-holomorphic type and $\sigma$ is an involutive automorphism of holomorphic type for $G=\mathrm{E}_{6(-14)}$.
\end{remark}
The following result is useful to the first step for the classification.
\begin{proposition}\label{13}
Let $X$ be a unitarizable simple $(\mathfrak{g},K)$-module, and let $\sigma\in\Gamma$. If $X$ is discretely decomposable as a $(\mathfrak{g}^\Gamma,K^\Gamma)$-module and $X$ is also discretely decomposable as a $(\mathfrak{g}^\sigma,K^\sigma)$-module, then there exists a unitarizable simple $(\mathfrak{g}^\sigma,K^\sigma)$-module that is discretely decomposable as a $(\mathfrak{g}^\Gamma,K^\Gamma)$-module.
\end{proposition}
\begin{proof}
See \cite[Lemma 9]{H3}.
\end{proof}
The author will find out all the Klein four symmetric pairs $(G,G^\Gamma)$ for $G=\mathrm{E}_{6(-14)}$ as required in \problemref{14}. Since the author excludes the cases when $G^\Gamma$ is compact, $\Gamma$ is not supposed to contain any Cartan involution of $G$. Moreover, due to the work in \cite{H3}, the involutive automorphism $\sigma$ in the second requirement of \problemref{14} is supposed to be of holomorphic type, and hence any highest / lowest weight simple $(\mathfrak{g},K)$-module is discretely decomposable as a $(\mathfrak{g}^\sigma,K^\sigma)$-module. Also, because of $(G,G^\Gamma)$ is assumed to be of non-holomorphic type, there is an involutive automorphism $\tau\in\Gamma$, which defines a symmetric pair of anti-holomorphic type. On the other hand, it is known from \cite{HM} that, for any nontrivial irreducible unitary representation of $G$, its restriction to its reductive subgroup contains no trivial subrepresentations. Hence in \propositionref{13}, if $X$ is supposed to be nontrivial, all the direct summands are nontrivial.

Based on the discussion as above, by the conditions in \problemref{14} and \propositionref{13}, the author needs to find reductive Lie group triples $(G,G^\sigma,G^\Gamma)$ such that
\begin{enumerate}[$\bullet$]
\item $G$ is either $\mathrm{E}_{6(-14)}$;
\item $G^\Gamma$ is noncompact;
\item $(G,G^\sigma)$ is a symmetric pair of holomorphic type;
\item $(G^\sigma,G^\Gamma)$ is a symmetric pair, such that there exists at least one nontrivial unitarizable simple $(\mathfrak{g}^\sigma,K^\sigma)$-module which is discretely decomposable as a $(\mathfrak{g}^\Gamma,K^\Gamma)$-module;
\item $G^\Gamma$ does not contain the center of $K$.
\end{enumerate}
According to \cite[Table C.2]{KO1}, there are totally four symmetric pairs $(G,G^\sigma)$ of holomorphic type with $G^\sigma$ noncompact, listed in the Lie algebra level $(\mathfrak{g},\mathfrak{g}^\sigma)$ as follows.
\begin{itemize}
\item $(\mathfrak{e}_{6(-14)},\mathfrak{so}(8,2)\oplus\mathfrak{so}(2))$
\item $(\mathfrak{e}_{6(-14)},\mathfrak{su}(4,2)\oplus\mathfrak{su}(2))$
\item $(\mathfrak{e}_{6(-14)},\mathfrak{so}^*(10)\oplus\mathfrak{so}(2))$
\item $(\mathfrak{e}_{6(-14)},\mathfrak{su}(5,1)\oplus\mathfrak{sl}(2,\mathbb{R}))$
\end{itemize}
For the time being, the following results ought to be well known and easy to prove.
\begin{proposition}\label{23}
Let $\mathfrak{g}=\mathfrak{g}_1\oplus\mathfrak{g}_2$ be a semisimple Lie algebra that is a direct sum of two non-isomorphic simple ideals $\mathfrak{g}_1$ and $\mathfrak{g}_2$. If $\varphi$ is a Lie algebra automorphism of $\mathfrak{g}$, then $\varphi\mathfrak{g}_1=\mathfrak{g}_1$ and $\varphi\mathfrak{g}_2=\mathfrak{g}_2$.
\end{proposition}
\begin{proof}
It is obvious that $\varphi\mathfrak{g}_1\cong\mathfrak{g}_1$ is an simple ideal of $\mathfrak{g}$, and hence $\varphi\mathfrak{g}_1\cap\mathfrak{g}_1$ is an ideal of $\mathfrak{g}$, which is either $\mathfrak{g}_1$ or $\{0\}$. Assume $\varphi\mathfrak{g}_1\cap\mathfrak{g}_1=\{0\}$, and then $\varphi\mathfrak{g}_1$ centralizes $\mathfrak{g}_1$. But the centralizer of $\mathfrak{g}_1$ in $\mathfrak{g}$ is exactly $\mathfrak{g}_2$, so $\varphi\mathfrak{g}_1\subseteq\mathfrak{g}_2$, which contradicts to the simplicity of $\mathfrak{g}_2$. Similarly, $\varphi\mathfrak{g}_2=\mathfrak{g}_2$.
\end{proof}
\begin{proposition}\label{24}
Let $\mathfrak{g}=Z(\mathfrak{g})\oplus[\mathfrak{g},\mathfrak{g}]$ be a reductive Lie algebra with the center $Z(\mathfrak{g})$. If $\varphi$ is a Lie algebra automorphism of $\mathfrak{g}$, then $\varphi Z(\mathfrak{g})=Z(\mathfrak{g})$ and $\varphi[\mathfrak{g},\mathfrak{g}]=[\mathfrak{g},\mathfrak{g}]$.
\end{proposition}
\begin{proof}
Any automorphism fixes the center, and then $\varphi Z(\mathfrak{g})=Z(\mathfrak{g})$. Moreover, $\varphi[\mathfrak{g},\mathfrak{g}]=[\varphi\mathfrak{g},\varphi\mathfrak{g}]=[\mathfrak{g},\mathfrak{g}]$.
\end{proof}
Suppose that $\mathfrak{g}_1$ and $\mathfrak{g}_2$ are two simple Lie algebras. It is well known that, if $X_i$ is a simple $(\mathfrak{g}_i,K_i)$-module for $i=1,2$, then $X_1\boxtimes X_2$ is a simple $(\mathfrak{g}_1\oplus\mathfrak{g}_2,K_1\times K_2)$-module. Conversely, any simple $(\mathfrak{g}_1\oplus\mathfrak{g}_2,K_1\times K_2)$-module is in the form of $X_1\boxtimes X_2$ where $X_i$ is a simple $(\mathfrak{g}_i,K_i)$-module for $i=1,2$.
\begin{proposition}\label{25}
Let $\mathfrak{g}_i$ for $i=1,2$ be two simple Lie algebras, and let $\mathfrak{g}'_i$ be a reductive subalgebra of $\mathfrak{g}_i$. Suppose that $X_i$ is a unitarizable simple $(\mathfrak{g}_i,K_i)$-module. Then $X_1\boxtimes X_2$ is discretely decomposable as a $(\mathfrak{g}'_1\oplus\mathfrak{g}'_2,K'_1\times K'_2)$-module if and only if $X_i$ is discretely decomposable as a $(\mathfrak{g}'_i,K'_i)$-module for $i=1,2$.
\end{proposition}
\begin{proof}
If $X_i$ is discretely decomposable as a $(\mathfrak{g}'_i,K'_i)$-module for $i=1,2$, by \propositionref{3}, there exists a simple $(\mathfrak{g}'_i,K'_i)$-module $X'_i$ such that $\mathrm{Hom}_{(\mathfrak{g}'_i,K'_i)}(X'_i,X_i)\neq\{0\}$ for $i=1,2$. Take a nonzero element $f_i\in\mathrm{Hom}_{(\mathfrak{g}'_i,K'_i)}(X'_i,X_i)$ for $i=1,2$, and $f_1\otimes f_2$ is a nonzero element in $\mathrm{Hom}_{(\mathfrak{g}'_1\oplus\mathfrak{g}'_2,K'_1\times K'_2)}(X'_1\boxtimes X'_2,X_1\boxtimes X_2)$. Again by \propositionref{3}, $X_1\boxtimes X_2$ is discretely decomposable as a $(\mathfrak{g}'_1\oplus\mathfrak{g}'_2,K'_1\times K'_2)$-module. Conversely, suppose that $X_1\boxtimes X_2$ is discretely decomposable as a $(\mathfrak{g}'_1\oplus\mathfrak{g}'_2,K'_1\times K'_2)$-module. By \propositionref{3}, there exists a simple $(\mathfrak{g}'_i,K'_i)$-module $X'_i$ for $i=1,2$ such that $\mathrm{Hom}_{(\mathfrak{g}'_1\oplus\mathfrak{g}'_2,K'_1\times K'_2)}(X'_1\boxtimes X'_2,X_1\boxtimes X_2)\neq\{0\}$. Take a nonzero element $v'\in X'_2$. Therefore, as a $(\mathfrak{g}'_1,K'_1)$-module, one has a nonzero injection given by $X'_1\cong X'_1\otimes v'\hookrightarrow X'_1\otimes X'_2\hookrightarrow X_1\otimes X_2\cong\bigoplus\limits_{v\in X_2}X_1\otimes v\cong\bigoplus\limits_{v\in X_2}X_1\hookrightarrow\prod\limits_{v\in X_2}X_1$. It follows that $\prod\limits_{v\in X_2}\mathrm{Hom}_{(\mathfrak{g}'_1,K'_1)}(X'_1,X_1)\cong \mathrm{Hom}_{(\mathfrak{g}'_1,K'_1)}(X'_1,\prod\limits_{v\in X_2}X_1)\neq\{0\}$, and hence $\mathrm{Hom}_{(\mathfrak{g}'_1,K'_1)}(X'_1,X_1)\neq\{0\}$. This shows that $X_1$ is discretely decomposable as a $(\mathfrak{g}'_1,K'_1)$-module by \propositionref{3}. Similarly, $X_2$ is discretely decomposable as a $(\mathfrak{g}'_2,K'_2)$-module.
\end{proof}
Now the author excludes some of the four symmetric pairs $(\mathfrak{g},\mathfrak{g}^\sigma)$ for $\mathfrak{g}=\mathfrak{e}_{6(-14)}$ listed above. Firstly, let $(\mathfrak{g},\mathfrak{g}^\sigma)=(\mathfrak{e}_{6(-14)},\mathfrak{so}^*(10)\oplus\mathfrak{so}(2))$. If $(\mathfrak{g},\mathfrak{g}^\Gamma)$ is a Klein four symmetric pair, then $\mathfrak{g}^\Gamma=(\mathfrak{g}^\sigma)^\tau=\mathfrak{so}^*(10)^\tau\oplus\mathfrak{so}(2)^\tau$ for some $\tau\in\Gamma$ by \propositionref{24}. If there is a simple $(\mathfrak{g}^\sigma,K^\sigma)$ module that is discretely decomposable as a $(\mathfrak{g}^\Gamma,K^\Gamma)$-module, then there exists a simple $(\mathfrak{so}^*(10),\mathrm{U}(5))$-module that is discretely decomposable as a $(\mathfrak{so}^*(10)^\tau,\mathrm{U}(5)^\tau)$-module by \propositionref{25}. According to \cite[Table C.2]{KO1} and \cite[Theorem 5.2 \& Table 1]{KO2}, such $(\mathfrak{so}^*(10),\mathfrak{so}^*(10)^\tau)$ must be a symmetric pair of holomorphic type; i.e., $\mathfrak{so}^*(10)^\tau$ contains the center of $\mathfrak{u}(5)$, the maximal compact subalgebra of $\mathfrak{so}^*(10)$. On the other hand, since $\mathfrak{g}^\sigma$ is not compact, the center of $\mathfrak{k}$ cannot centralize the whole $\mathfrak{g}^\sigma$. It follows that the center of $\mathfrak{k}$ does not come from $\mathfrak{so}(2)$, and is equal to the center of $\mathfrak{u}(5)$. Thus, $\mathfrak{g}^\Gamma$ contains the center of $\mathfrak{k}$, which contradicts to the requirement that $G^\Gamma$ does not contain the center of $K$. Hence, there is no need to consider $(\mathfrak{g},\mathfrak{g}^\sigma)=(\mathfrak{e}_{6(-14)},\mathfrak{so}^*(10)\oplus\mathfrak{so}(2))$.

Secondly, let $(\mathfrak{g},\mathfrak{g}^\sigma)=(\mathfrak{g},\mathfrak{g}_1\oplus\mathfrak{g}_2)=(\mathfrak{e}_{6(-14)},\mathfrak{su}(5,1)\oplus \mathfrak{sl}(2,\mathbb{R}))$. If $(\mathfrak{g},\mathfrak{g}^\Gamma)$ is a Klein four symmetric pair, by \propositionref{23}, $\mathfrak{g}^\Gamma=(\mathfrak{g}^\sigma)^\tau=\mathfrak{g}_1^\tau\oplus\mathfrak{g}_2^\tau$ for some $\tau\in\Gamma$. If $X_1\boxtimes X_2$ is a simple $(\mathfrak{g}^\sigma,K^\sigma)$ module that is discretely decomposable as a $(\mathfrak{g}^\Gamma,K^\Gamma)$-module, $X_i$ is discretely decomposable as a $(\mathfrak{g}_i^\tau,K_i^\tau)$-module by \propositionref{25}. According to \cite[Table C.2]{KO1} and \cite[Theorem 5.2 \& Table 1]{KO2}, for $\mathfrak{g}_1=\mathfrak{su}(5,1)$ or $\mathfrak{g}_2=\mathfrak{sl}(2,\mathbb{R})$, if there exists a simple $(\mathfrak{g}_i,K_i)$-module that is discretely decomposable as a $(\mathfrak{g}_i^\tau,K_i^\tau)$-module, $(\mathfrak{g}_i,\mathfrak{g}_i^\tau)$ must be a symmetric pair of holomorphic type. Hence, $\mathfrak{g}^\Gamma$ contains the center of $\mathfrak{k}^\sigma$. Since $(\mathfrak{g},\mathfrak{g}^\sigma)$ is supposed to be of holomorphic type, the center of $\mathfrak{k}^\sigma$ contains that of $\mathfrak{k}$. But this does not satisfy the requirement that $G^\Gamma$ does not contain the center of $K$. Hence, there is no need to consider $(\mathfrak{g},\mathfrak{g}^\sigma)=(\mathfrak{e}_{6(-14)},\mathfrak{su}(5,1)\oplus\mathfrak{sl}(2,\mathbb{R}))$.

Thus, there remain two symmetric pairs of holomorphic type $(\mathfrak{g},\mathfrak{g}^\sigma)$.
\begin{itemize}
\item $(\mathfrak{e}_{6(-14)},\mathfrak{so}(8,2)\oplus\mathfrak{so}(2))$
\item $(\mathfrak{e}_{6(-14)},\mathfrak{su}(4,2)\oplus\mathfrak{su}(2))$
\end{itemize}
By the similar argument as in the case of $(\mathfrak{e}_{6(-14)},\mathfrak{so}^*(10)\oplus\mathfrak{so}(2))$, the center of the maximal compact subalgebra of $\mathfrak{e}_{6(-14)}$ is contained in $\mathfrak{so}(8,2)$ and $\mathfrak{su}(4,2)$ for $(\mathfrak{g},\mathfrak{g}^\sigma)=(\mathfrak{e}_{6(-14)},\mathfrak{so}(8,2)\oplus\mathfrak{so}(2))$ and $(\mathfrak{e}_{6(-14)},\mathfrak{su}(4,2)\oplus\mathfrak{su}(2))$ respectively. Notice that $(\mathfrak{g}^\sigma,\mathfrak{g}^\Gamma)$ is a symmetric pair and $\mathfrak{g}^\Gamma$ does not contain the center of $\mathfrak{k}$. By \propositionref{23}, \propositionref{24}, \cite[Table C.2]{KO1}, and \cite[Table 1]{KO2}, the possible triples $(\mathfrak{g},\mathfrak{g}^\sigma,\mathfrak{g}^\Gamma)$ are
\begin{eqnarray*}
\begin{array}{lllllll}
&&(\mathfrak{e}_{6(-14)},\mathfrak{so}(8,2)\oplus\mathfrak{so}(2),\mathfrak{so}(8,1)\oplus\mathfrak{so}(2)), &&(\mathfrak{e}_{6(-14)},\mathfrak{so}(8,2)\oplus\mathfrak{so}(2),\mathfrak{so}(8,1));\\
&&(\mathfrak{e}_{6(-14)},\mathfrak{su}(4,2)\oplus\mathfrak{su}(2),\mathfrak{sp}(2,1)\oplus\mathfrak{su}(2)), &&(\mathfrak{e}_{6(-14)},\mathfrak{su}(4,2)\oplus\mathfrak{su}(2),\mathfrak{sp}(2,1)\oplus\mathfrak{so}(2)).
\end{array}
\end{eqnarray*}
The author does not need to consider $(\mathfrak{e}_{6(-14)},\mathfrak{so}(8,2)\oplus\mathfrak{so}(2),\mathfrak{so}(8,1))$ because of the following result.
\begin{proposition}\label{31}
For $(\mathfrak{g},\mathfrak{g}^\sigma,\mathfrak{g}^\Gamma)=(\mathfrak{e}_{6(-14)},\mathfrak{so}(8,2)\oplus\mathfrak{so}(2),\mathfrak{so}(8,1))$, $(\mathfrak{g},\mathfrak{g}^\Gamma)$ is a Klein four symmetric pair. Moreover, there exists a unitarizable simple $(\mathfrak{g},K)$-module that is discretely decomposable as a $(\mathfrak{g}^\sigma,K^\sigma)$-module and is discretely decomposable as a $(\mathfrak{g}^\Gamma,K^\Gamma)$-module.
\end{proposition}
\begin{proof}
See \cite[Proposition 10 \& Proposition 13]{H3}.
\end{proof}
In the next two subsections, the author will determine whether each of the remaining there pairs $(\mathfrak{g},\mathfrak{g}^\Gamma)$ is really a Klein four symmetric pair.
\subsection{Elementary abelian 2-subgroups of $\mathrm{Aut}\mathfrak{e}_{6(-78)}$}
Let $\mathfrak{e}_6$ be the complex simple Lie algebra of type $\mathrm{E}_6$. Fix a Cartan subalgebra of $\mathfrak{e}_6$ and a simple root system $\{\alpha_i\mid1\leq i\leq6\}$, the Dynkin diagram of which is given in Figure 1.
\begin{figure}
\centering \scalebox{0.7}{\includegraphics{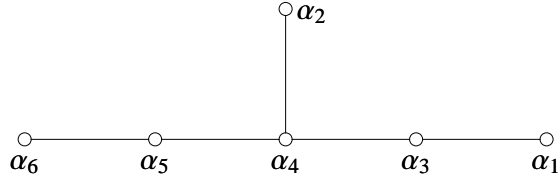}}
\caption{Dynkin diagram of $\mathrm{E}_6$.}
\end{figure}
For each root $\alpha$, denote by $H_\alpha$ its coroot, and denote by $X_\alpha$ the normalized root vector so that $[X_\alpha,X_{-\alpha}]=H_\alpha$. Moreover, one can normalize $X_\alpha$ appropriately such that\[\mathrm{Span}_\mathbb{R}\{X_\alpha-X_{-\alpha},\sqrt{-1}(X_\alpha+X_{-\alpha}),\sqrt{-1}H_\alpha\mid\alpha:\textrm{positive root}\}\cong\mathfrak{e}_{6(-78)}\]is a compact real form of $\mathfrak{e}_6$ by \cite{Kn}. It is well known that\[\mathrm{Aut}\mathfrak{e}_{6(-78)}/\mathrm{Int}\mathfrak{e}_{6(-78)}\cong\mathrm{Aut}\mathfrak{e}_6/\mathrm{Int}\mathfrak{e}_6\]which is just the automorphism group of the Dynkin diagram.

Follow the constructions of involutive automorphisms of $\mathfrak{e}_{6(-78)}$ in \cite{Y}. Let $\omega$ be the specific involutive automorphism of the Dynkin diagram defined by
\begin{eqnarray*}
\begin{array}{rclcrcl}
\omega(H_{\alpha_1})=H_{\alpha_6},&&\omega(X_{\pm\alpha_1})=X_{\pm\alpha_6},\\
\omega(H_{\alpha_2})=H_{\alpha_2},&&\omega(X_{\pm\alpha_2})=X_{\pm\alpha_2},\\
\omega(H_{\alpha_3})=H_{\alpha_5},&&\omega(X_{\pm\alpha_3})=X_{\pm\alpha_5},\\
\omega(H_{\alpha_4})=H_{\alpha_4},&&\omega(X_{\pm\alpha_4})=X_{\pm\alpha_4}.
\end{array}
\end{eqnarray*}
Let $\sigma_1=\mathrm{exp}(\sqrt{-1}\pi H_{\alpha_2})$, $\sigma_2=\mathrm{exp}(\sqrt{-1}\pi(H_{\alpha_1}+H_{\alpha_6}))$, $\sigma_3=\omega$, $\sigma_4=\omega\mathrm{exp}(\sqrt{-1}\pi H_{\alpha_2})$, where $\mathrm{exp}$ represents the exponential map from $\mathfrak{e}_{6(-78)}$ to $\mathrm{Aut}\mathfrak{e}_{6(-78)}$. Then $\sigma_1$, $\sigma_2$, $\sigma_3$, and $\sigma_4$ represent all conjugacy classes of involutive automorphisms in $\mathrm{Aut}\mathfrak{e}_{6(-78)}$, which correspond to real forms $\mathfrak{e}_{6(2)}$, $\mathfrak{e}_{6(-14)}$, $\mathfrak{e}_{6(-26)}$, and $\mathfrak{e}_{6(6)}$ respectively.

From \cite{Y}, it is known that $(\mathrm{Int}\mathfrak{e}_{6(-78)})^{\sigma_3}\cong\mathrm{F}_{4(-52)}$, the compact Lie group of type $\mathrm{F}_4$, and there exist involutive automorphisms $\tau_1$ and $\tau_2$ of $\mathrm{F}_{4(-52)}$ such that $\mathfrak{f}_{4(-52)}^{\tau_1}\cong\mathfrak{sp}(3)\oplus\mathfrak{sp}(1)$ and $\mathfrak{f}_{4(-52)}^{\tau_2}\cong\mathfrak{so}(9)$, where $\mathfrak{f}_{4(-52)}$ denotes the compact Lie algebra of type $\mathrm{F}_4$. Moreover, $\tau_1$ and $\tau_2$ represent all conjugacy classes of involutive automorphisms in $\mathrm{Aut}\mathfrak{f}_{4(-52)}$. Now, consider $((\mathrm{Int}\mathfrak{e}_{6(-78)})^{\sigma_3})^{\tau_1}\cong\mathrm{Sp}(3)\times\mathrm{Sp}(1)/\langle(-I_3,-1)\rangle$. Let $\mathbf{i}$, $\mathbf{j}$, and $\mathbf{k}$ denote the fundamental quaternion units, and then set $x_0=\sigma_3$, $x_1=\tau_1=(I_3,-1)$, $x_2=(\mathbf{i}I_3,\mathbf{i})$, $x_3=(\mathbf{j}I_3,\mathbf{j})$, $x_4=(\left(\begin{array}{ccc}-1&0&0\\0&-1&0\\0&0&1\end{array}\right),1)$, and $x_5=(\left(\begin{array}{ccc}-1&0&0\\0&1&0\\0&0&-1\end{array}\right),1)$.

For a pair $(r,s)$ of integers with $r\leq2$ and $s\leq3$, define\[F_{r,s}:=\langle x_0,x_1,\cdots,x_s,x_4,x_5,\cdots,x_{r+3}\rangle\]and\[F'_{r,s}:=\langle x_1,x_2,\cdots,x_s,x_4,x_5,\cdots,x_{r+3}\rangle.\]
\begin{remark}\label{28}
Besides the two families $F_{r,s}$ and $F'_{r,s}$ of the elementary abelian 2-subgroups for $\mathrm{Aut}\mathfrak{e}_{6(-78)}$ through $\sigma_3$, Jun YU also defined another two families $F_{u,v,r,s}$ and $F'_{u,v,r,s}$ through $\sigma_4$ in \cite{Y}. According to \cite[Proposition 6.3 \& Proposition 6.5]{Y}, each elementary abelian 2-subgroup in $\mathrm{Aut}\mathfrak{e}_{6(-78)}$ is conjugate to one of the groups in the one family of $F_{r,s}$, $F'_{r,s}$, $F_{u,v,r,s}$, and $F'_{u,v,r,s}$, and all groups in the families of $F_{r,s}$, $F'_{r,s}$, $F_{u,v,r,s}$, and $F'_{u,v,r,s}$ are pairwisely non-conjugate. In this article, $F_{u,v,r,s}$ and $F'_{u,v,r,s}$ will not be used, so the author does not write down the detailed constructions for them.
\end{remark}
\subsection{Klein four symmetric pairs for $\mathfrak{e}_{6(-14)}$}
If $\sigma$ is an involutive automorphism of $\mathfrak{e}_{6(-78)}$, which is conjugate to $\sigma_2$ as defined above, then by the construction, $\mathfrak{e}_{6(-14)}\cong\mathfrak{e}_{6(-78)}^\sigma+\sqrt{-1}\mathfrak{e}_{6(-78)}^{-\sigma}$, where $\mathfrak{e}_{6(-78)}^{\pm\sigma}$ are the eigenspaces of eigenvalues $\pm1$ under the action of $\sigma$ on $\mathfrak{e}_{6(-78)}$. Thus, if $\langle\sigma,\mu_1,\mu_2,\cdots,\mu_n\rangle$ is an elementary abelian 2-subgroups of rank $n+1$ of $\mathrm{Aut}\mathfrak{e}_{6(-78)}$, then $\langle\mu_1,\mu_2,\cdots,\mu_n\rangle$ is an elementary abelian 2-subgroups of rank $n$ of $\mathrm{Aut}\mathfrak{e}_{6(-14)}$. Conversely, if $\langle\eta_1,\eta_2,\cdots,\eta_n\rangle$ is an elementary abelian 2-subgroups of rank $n$ of $\mathrm{Aut}\mathfrak{e}_{6(-14)}$, by \cite[Proposition 1]{H1} there exists a Cartan involution $\sigma\in\mathrm{Aut}\mathfrak{e}_{6(-14)}$ such that $\langle\sigma,\eta_1,\eta_2,\cdots,\eta_n\rangle$ is an elementary abelian 2-subgroups of rank $n+1$ of $\mathrm{Aut}\mathfrak{e}_{6(-78)}$. Under this correspondence, $\mathfrak{e}_{6(-78)}^{\langle\eta_1,\eta_2,\cdots,\eta_n\rangle}$ is the compact dual of $\mathfrak{e}_{6(-14)}^{\langle\eta_1,\eta_2,\cdots,\eta_n\rangle}$ whose maximal compact subalgebra is $\mathfrak{e}_{6(-14)}^{\langle\sigma,\eta_1,\eta_2,\cdots,\eta_n\rangle}\cong\mathfrak{e}_{6(-78)}^{\langle\sigma,\eta_1,\eta_2,\cdots,\eta_n\rangle}$.
\begin{lemma}\label{26}
The pair $(\mathfrak{e}_{6(-14)},\mathfrak{sp}(2,1)\oplus\mathfrak{su}(2))$ is a Klein four symmetric pair, while neither $(\mathfrak{e}_{6(-14)},\mathfrak{so}(8,1)\oplus\mathfrak{so}(2))$ nor $(\mathfrak{e}_{6(-14)},\mathfrak{sp}(2,1)\oplus\mathfrak{so}(2))$ is.
\end{lemma}
\begin{proof}
Retain all the notations as in Section 4.2. For the first statement, consider the subgroup $F_{1,1}=\langle x_0,x_1,x_4\rangle$ of $\mathrm{Aut}\mathfrak{e}_{6(-78)}$. It is known from \cite{Y} and \cite[Lemma 8]{H1} that $x_0$ is conjugate to $\sigma_3$, $x_1$ is conjugate to $\sigma_1$, and $x_4$ is conjugate to $\sigma_2$. Then $\mathfrak{e}_{6(-14)}\cong\mathfrak{e}_{6(-78)}^{x_4}+\sqrt{-1}\mathfrak{e}_{6(-78)}^{-x_4}$. According to \cite[Table 3]{Y}, up to conjugation, the only Klein four subgroup of $\mathrm{Aut}\mathfrak{e}_{6(-78)}$, which contains both an element conjugate to $\sigma_1$ and an element conjugate to $\sigma_3$, is $\Gamma_5$ (using the notation in \cite[Table 3]{Y}), so $\mathfrak{e}_{6(-78)}^{\langle x_0,x_1\rangle}\cong\mathfrak{e}_{6(-78)}^{\Gamma_5}\cong\mathfrak{sp}(3)\oplus\mathfrak{su}(2)$ is the compact dual of $\mathfrak{e}_{6(-14)}^{\langle x_0,x_1\rangle}$. On the other hand, $\mathfrak{e}_{6(-14)}^{\langle x_0,x_1,x_4\rangle}=\mathfrak{e}_{6(-78)}^{\langle x_0,x_1,x_4\rangle}$ is the maximal compact subalgebra of $\mathfrak{e}_{6(-14)}^{\langle x_0,x_1\rangle}$. Notice that $\mathfrak{e}_{6(-78)}^{x_0}\cong\mathfrak{f}_{4(-52)}$. Moreover, $x_1$ is conjugate to $\tau_1$ and $x_4$ is conjugate to $\tau_2$ in $\mathrm{Aut}\mathfrak{f}_{4(-52)}$. Again by \cite[Table 3]{Y}, up to conjugation, the only Klein four subgroup of $\mathrm{Aut}\mathfrak{f}_{4(-52)}$, which contains both an element conjugate to $\tau_1$ and an element conjugate to $\tau_2$, is $\Gamma_2$ (using the notation in \cite[Table 3]{Y}), so $\mathfrak{e}_{6(-14)}^{\langle x_0,x_1,x_4\rangle}=(\mathfrak{e}_{6(-78)}^{x_0})^{\langle x_1,x_4\rangle}\cong\mathfrak{f}_{4(-52)}^{\langle x_1,x_4\rangle}\cong\mathfrak{f}_{4(-52)}^{\Gamma_2}\cong\mathfrak{so}(5)\oplus2\mathfrak{su}(2)$ is the maximal compact subalgebra of $\mathfrak{e}_{6(-14)}^{\langle x_0,x_1\rangle}$. Thus, one obtains $\mathfrak{e}_{6(-14)}^{\langle x_0,x_1\rangle}\cong\mathfrak{sp}(2,1)\oplus\mathfrak{su}(2)$. For the second statement, it follows immediately from the fact that there is no Klein four subgroups $\Gamma\subseteq\mathrm{Aut}\mathfrak{e}_{6(-78)}$ such that $\mathfrak{e}_{6(-78)}^\Gamma$ is conjugate to $\mathfrak{so}(9)\oplus\mathfrak{so}(2)$ or $\mathfrak{sp}(3)\oplus\mathfrak{so}(2)$, the compact dual of $\mathfrak{so}(8,1)\oplus\mathfrak{so}(2)$ or $\mathfrak{sp}(2,1)\oplus\mathfrak{so}(2)$, according to \cite[Table 3]{Y}.
\end{proof}
\subsection{Discrete decomposability for $(\mathfrak{e}_{6(-14)},\mathfrak{sp}(2,1)\oplus\mathfrak{su}(2))$}
In the final subsection, the author focuses on the Klein four symmetric pair $(\mathfrak{e}_{6(-14)},\mathfrak{sp}(2,1)\oplus\mathfrak{su}(2))$. In order to apply \theoremref{10} or \corollaryref{12} to the pair, the author needs to know how each nonidentity element behaves on $\mathfrak{e}_{6(-14)}$. Retain the notations as in Subsection 4.2, it is known from the proof for \lemmaref{26} that $\mathfrak{e}_{6(-14)}\cong\mathfrak{e}_{6(-78)}^{x_4}+\sqrt{-1}\mathfrak{e}_{6(-78)}^{-x_4}$ and $\Gamma=\langle x_0,x_1\rangle$.
\begin{lemma}\label{32}
One has $\mathfrak{e}_{6(-14)}^{x_0}\cong\mathfrak{f}_{4(-20)}$, $\mathfrak{e}_{6(-14)}^{x_1}\cong\mathfrak{su}(4,2)\oplus\mathfrak{su}(2)$, and $\mathfrak{e}_{6(-14)}^{x_0x_1}\cong\mathfrak{sp}(2,2)$.
\end{lemma}
\begin{proof}
Since $x_0$ is conjugate to $\sigma_3$ in $\mathrm{Aut}\mathfrak{e}_{6(-14)}$, it follows from \cite[Table 2]{Y} that $\mathfrak{e}_{6(-78)}^{x_0}\cong\mathfrak{f}_{4(-52)}$, which is the compact dual of $\mathfrak{e}_{6(-14)}^{x_0}$. Moreover, since $x_4$ is conjugate to $\sigma_2$, according to \cite[Table 3]{Y}, up to conjugation the only Klein four subgroup of $\mathrm{Aut}\mathfrak{e}_{6(-14)}$, which contains both an element conjugate to $\sigma_2$ and an element conjugate to $\sigma_3$, is $\Gamma_7$ (using the notation in \cite[Table 3]{Y}). Thus, $\mathfrak{e}_{6(-14)}^{\langle x_0,x_4\rangle}\cong\mathfrak{so}(9)$ which is the maximal compact subalgebra of $\mathfrak{e}_{6(-14)}^{x_0}$. Therefore, $\mathfrak{e}_{6(-14)}^{x_0}\cong\mathfrak{f}_{4(-20)}$.

Since $x_1$ is conjugate to $\sigma_1$ in $\mathrm{Aut}\mathfrak{e}_{6(-14)}$, it follows from \cite[Table 2]{Y} that $\mathfrak{e}_{6(-78)}^{x_1}\cong\mathfrak{su}(6)\oplus\mathfrak{su}(2)$, which is the compact dual of $\mathfrak{e}_{6(-14)}^{x_1}$. Moreover, since $x_4$ is conjugate to $\sigma_2$ and $x_1x_4$ is conjugate to $\sigma_1$ in $\mathrm{Aut}\mathfrak{e}_{6(-14)}$ by \cite[Lemma 8]{H1}, according to \cite[Table 3]{Y}, $\mathfrak{e}_{6(-14)}^{\langle x_1,x_4\rangle}=\mathfrak{e}_{6(-78)}^{\langle x_1,x_4\rangle}\cong\mathfrak{su}(4)\oplus2\mathfrak{su}(2)\oplus\mathfrak{so}(2)$ which is the maximal compact subalgebra of $\mathfrak{e}_{6(-14)}^{x_1}$. Therefore, $\mathfrak{e}_{6(-14)}^{x_1}\cong\mathfrak{su}(4,2)\oplus\mathfrak{su}(2)$.

It is known from \cite{Y} that $x_0x_1=\sigma_3\tau_1$ is conjugate to $\sigma_4$, and hence $\mathfrak{e}_{6(-78)}^{x_0x_1}\cong\mathfrak{sp}(4)$ by \cite[Table 2]{Y}, which is the compact dual of $\mathfrak{e}_{6(-14)}^{x_0x_1}$. Moreover, since $x_4$ is conjugate to $\sigma_2$, according to \cite[Table 3]{Y}, up to conjugation the only Klein four subgroup of $\mathrm{Aut}\mathfrak{e}_{6(-14)}$, which contains both an element conjugate to $\sigma_2$ and an element conjugate to $\sigma_4$, is $\Gamma_8$ (using the notation in \cite[Table 3]{Y}). Thus, $\mathfrak{e}_{6(-14)}^{\langle x_0x_1,x_4\rangle}\cong2\mathfrak{so}(5)\cong2\mathfrak{sp}(2)$ which is the maximal compact subalgebra of $\mathfrak{e}_{6(-14)}^{x_0}$. Therefore, $\mathfrak{e}_{6(-14)}^{x_0x_1}\cong\mathfrak{sp}(2,2)$.
\end{proof}
\begin{lemma}\label{17}
There does not exist any nontrivial unitarizable simple $(\mathfrak{e}_{6(-14)},\mathrm{SO}(10)\times\mathrm{SO}(2))$-module that is discretely decomposable as a $(\mathfrak{sp}(2,1)\oplus\mathfrak{su}(2),\mathrm{Sp}(2)\times\mathrm{SU}(2)^2)$-module.
\end{lemma}
\begin{proof}
It is known from \cite[Theorem 5.2 \& Corollary 5.8 \& Table 1]{KO2} that there exists a nontrivial unitarizable simple $(\mathfrak{e}_{6(-14)},\mathrm{SO}(10)\times\mathrm{SO}(2))$-module which is discretely decomposable as a $(\mathfrak{f}_{4(-20)},\mathfrak{so}(9))$-module, that there exists a nontrivial unitarizable simple $(\mathfrak{e}_{6(-14)},\mathrm{SO}(10)\times\mathrm{SO}(2))$-module which is discretely decomposable as a $(\mathfrak{su}(4,2)\oplus\mathfrak{su}(2),\mathrm{SU}(2)^2\times\mathrm{SO}(2))$-module, and that there does not exist any nontrivial unitarizable simple $(\mathfrak{e}_{6(-14)},\mathrm{SO}(10)\times\mathrm{SO}(2))$-module which is discretely decomposable as a $(\mathfrak{sp}(2,2),\mathrm{Sp}(2)^2)$-module. Therefore, the conclusion follows from \lemmaref{32} and \corollaryref{12}.
\end{proof}
Now the author can solve \problemref{14} for $G=\mathrm{E}_{6(-14)}$.
\begin{theorem}\label{18}
Let $G=\mathrm{E}_{6(-14)}$ and $(G,G^\Gamma)$ a Klein four symmetric pair with $G^\Gamma$ noncompact. Then there exists a nontrivial unitarizable simple $(\mathfrak{g},K)$-module $X$ that is discretely decomposable as a $(\mathfrak{g}^\Gamma,K^\Gamma)$-module and is also discretely decomposable as a $(\mathfrak{g}^\sigma,K^\sigma)$-module for some nonidentity element $\sigma\in\Gamma$, if and only if $(\mathfrak{g},\mathfrak{g}^\Gamma)=(\mathfrak{e}_{6(-14)},\mathfrak{so}(8,1))$.
\end{theorem}
\begin{proof}
If $\sigma$ is such that $(\mathfrak{g},\mathfrak{g}^\sigma)$ is a symmetric pair of holomorphic type, the conclusion follows from \propositionref{31} and \lemmaref{17}. If $\sigma$ is such that $(\mathfrak{g},\mathfrak{g}^\sigma)$ is a symmetric pair of anti-holomorphic type, the conclusion follows from \cite[Proposition 13 \& Proposition 15]{H3}.
\end{proof}

\end{document}